\documentclass[reqno,11pt]{amsart}

\pdfoutput=1 

\numberwithin{equation}{section}

\usepackage{verbatim} 
\usepackage{xspace} 
\usepackage{amssymb}
\usepackage{amsmath}
\usepackage{amsthm} %
\usepackage[latin1]{inputenc}
\usepackage{graphicx}
\usepackage{amscd,amssymb,euscript,graphics}
\makeindex

\newcommand{\Z}{\ensuremath{\mathbb{Z}}}

\newtheorem{theorem}{Theorem}[section]

\newtheorem{lem}[theorem]{Lemma}
\newtheorem{cor}[theorem]{Corollary}

\newcount\theTime
\newcount\theHour
\newcount\theMinute
\newcount\theMinuteTens
\newcount\theScratch
\theTime=\number\time
\theHour=\theTime
\divide\theHour by 60
\theScratch=\theHour
\multiply\theScratch by 60
\theMinute=\theTime
\advance\theMinute by -\theScratch
\theMinuteTens=\theMinute
\divide\theMinuteTens by 10
\theScratch=\theMinuteTens
\multiply\theScratch by 10
\advance\theMinute by -\theScratch

\def\today{{\number\day\space
 \ifcase\month\or
  January\or February\or March\or April\or May\or June\or
  July\or August\or September\or October\or November\or December\fi
 \space\number\year}}

\begin{document}

\title[Scaling of Congestion in Small World Networks]{Scaling of Congestion in Small World Networks}

\author{Iraj Saniee and Gabriel H. Tucci} 
\thanks{I. Saniee and G. H. Tucci are with Bell Laboratories, Alcatel-Lucent, 600 Mountain Avenue, Murray Hill, New Jersey 07974, USA. 
E-mail: iis@research.bell-labs.com, gabriel.tucci@alcatel-lucent.com}

\begin{abstract}
In this report we show that in a planar exponentially growing network
consisting of $N$ nodes, 
congestion scales as $O(N^2/\log(N))$ independently of how flows may be routed. 
This is in contrast to the $O(N^{3/2})$ scaling of congestion in a flat polynomially growing network. 
We also show that without the planarity condition,
congestion in a small world network could scale as low as $O(N^{1+\epsilon})$, for arbitrarily small $\epsilon$. 
These extreme results demonstrate that the small world property by itself cannot provide
guidance on the level of congestion in a network and
other characteristics are needed for better resolution. Finally, we investigate 
scaling of congestion under the geodesic flow, that is,
when flows are routed on shortest paths based on a link metric.
Here we prove that if the link weights 
are scaled by arbitrarily small or large multipliers
then considerable changes in congestion may occur. However, if we 
constrain the link-weight multipliers to be bounded away from both
zero and infinity, then variations in congestion due to such remetrization
are negligible.

\end{abstract}

\maketitle

\section{Introduction}

\par The study of large-scale (complex) networks, such as computer, biological and social networks, is a multidisciplinary field that combines ideas from mathematics, physics, biology, social sciences and other fields. A remarkable and widely discussed phenomena associated with such networks is the {\it small world} property.   It is observed in many such 
networks, man-made or natural, that the typical distance between the nodes is surprisingly small. 
More formally, as a function of
the number of nodes, $N$, the average distance between a node pair 
typically scales at or below $O(\log(N))$.  


\par In this work, we study the load characteristics of small world networks.  Assuming one unit of 
demand between each node pair, we quantify as a function of $N$, how the maximal nodal load scales, independently of how 
each unit of demand may be routed. In other words, we are interested in the smallest of
such maximal nodal loads as a function of routing, which we refer to as {\it congestion}, that the network could 
experience.  
We show that in the planar small-world network congestion is almost quadratic in $N$, which is as high as
it can get, specifically $O(N^2/\log(N))$.  
In contrast, for some non-planar small-world networks, congestion may be 
almost linear in $N$, namely $O(N^{1+\epsilon})$ for arbitrarily small $\epsilon$.  
Since congestion in a network with $N$ nodes cannot have scaling order less than $O(N)$ or more than $O(N^2)$,
we conclude that the small world property alone is not sufficient to predict the level of congestion 
{\it a priori} and additional
characteristics may be needed to explain congestion features of complex networks. This
has been argued in 
\cite{NS2, NS1, Tucci1, Tucci2, H0, H3, H4, jonck} for the case of intrinsic hyperbolicity, which is 
a geometric feature above and beyond the small world property.

\par Additionally, we investigate what happens to congestion when we change the link metric that
prescribes routing. That is, for a network with edge weight $\{d_e\}_{e\in E}$ we change the metric 
by a factor $0\leq w_e \leq \infty$
thus assigning each edge $e$ a new weight $w_ed_e$. We explore the extent to which this change in the metric 
can change congestion in the network. We prove that if we allow the weights to get arbitrarily small or large, 
i.e. when for some edges, $w_e$ approach zero and for some others $w_e$ approach infinity, 
then considerable changes in congestion can 
occur. On the other hand, if we require the weights to be bounded away from zero 
and infinity, i.e. when $0 < k \leq w_e \leq K < \infty$ for all edges $e$, 
then congestion cannot change significantly.  These observations
quantify the degree to which remetrization in a small world network may be helpful in affecting congestion.

\section{Traffic on Small World Planar Graphs}\label{sec2}

As mentioned in the introduction, the small world property is ubiquitous in complex networks. 
Formally we say that a graph has the small world property if its diameter $D$ is of the order 
$\log(N)$ where $N$ is the number of nodes in the graph. It has been shown that a
surprising number of real-life, man-made or natural,
networks have the small world property, see, for example, \cite{Newman1, Newman2}. 
To be more specific, assume 
$G$ is an infinite planar graph and let $x_0$ be an arbitrary fixed node in the graph that 
we shall designate as the root. Let us assign a weight $1$ to each edge, thus $d_e =1$ for all edges $e$, and let $G_n=B(x_0,n)$ be the ball of center $x_0$ and radius $n$. In other words a node 
$y$ belongs to $G_n$ if and only of $d(x_0,y)\leq n$. 
More generally, we can consider weighted graphs where each edge has a non-negative length. 
We will further assume that the sub-graphs $G_{n}$ have exponential growth, which is clearly 
equivalent to the small world property, as defined above. More precisely, there exist $n_0$ and $\lambda>1$ 
such that
$$
|G_{n}|\geq \lambda^{n}
$$
for all $n\geq n_0$.

\par Assume that for every $n$ and every pair of nodes in $G_{n}$ 
there is a unit of demand between each node pair. Therefore, the total demand in $G_{n}$ is $T_{n}(G_{n})=N(N-1)/2$ 
where $N=N(n)=|G(n)|$. Given a node $v$ in $G_{n}$ we denote by $T_{n}(v)$ the total flow 
in $G_{n}$ routed through $v$. The load, or congestion, $T_{n}^{\max} (R)$, is the maximum of $T_n(v)$ over all vertices
$v$ in $G$, which is typically a function of the routing $R$.

The next theorem shows that for a planar graph with exponential growth, there exist 
nodes with load $O(N^2/\log(N))$ for $n$ sufficiently large regardless of routing. 

\begin{figure}[!Ht]\label{Fig1}
  \begin{center}
    \includegraphics[width=2.3in]{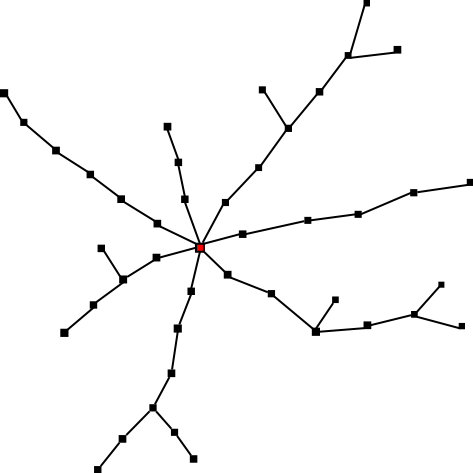}\hspace{2cm}
    \includegraphics[width=1.8in]{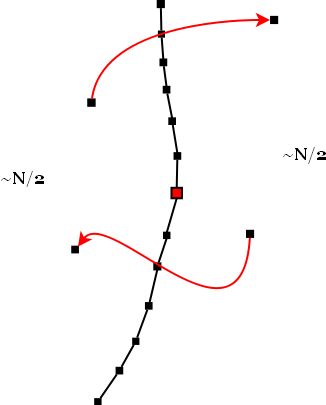}
	\caption{Left: A simplified figure of a minimal spanning tree with root $x_0$. Right: Flows crossing the boundary 
	separating two planar wedges each with $N/2 \pm c *\log(N)$ nodes. Note that (red) geodesic paths may cross 
	the boundary more than once.}
  \end{center}
\end{figure}

\begin{theorem}
Let $G$ be an infinite planar graph with exponential growth.
Assume one unit of demand between every pair of nodes in $G_{n}$.
Then for every $n$ there exists a node $v\in G_{n}$ such that 
$$
T_{n}(v)\geq \frac{N^2}{16n}-\frac{N}{8}
$$
where $N=|G_n|$.
\end{theorem}

\begin{proof}
Fix $n$ and let $\mathcal{P}_{n}$ be the spanning tree of all geodesic (shortest) paths with the node $x_0$ as the origin, as shown in
Figure~1, left. 
Observe that since $G$ is small world, all node pairs have distance $O(\log(N))$ and thus each ray from $x_0$ in $\mathcal{P}_{n}$ has length thus bounded.

Enumerate all the paths in $\mathcal{P}_{n}$ in clockwise order, possible because of the planarity of $G_n$. Therefore, 
$$
\mathcal{P}_{n}=\{\gamma_1,\gamma_2,\ldots, \gamma_M\}
$$
and by the small world property each ray $|\gamma_i | \leq O(log(N))$.  Let $W_1=\gamma_1$, $W_2=\gamma_1\cup\gamma_2$ and in general $W_i=\gamma_1\cup \ldots \cup\gamma_i$ for $i\in\{1,2,\ldots, M\}$. It is clear that for all $i$
\begin{equation}
\label{eqq1}
|W_{i+1}|-|W_{i}|\leq |\gamma_{i+1}|\leq n.
\end{equation}

We also know that
\begin{equation}
\label{eqq2}
|W_1|\leq |W_2|\leq \ldots \leq |W_{M}|=|G_{n}|=N
\end{equation}

There exists $i_0$ such that $|W_{i_0}|\geq N/2$ and $|W_{i_0-1}|< N/2$ because addition of
each $\gamma_i$ adds at most $O(\log(N))$ nodes to $W_i$, and moreover by inequality 
(\ref{eqq1}) we know that
$$
\frac{N}{2}\leq |W_{i_{0}}|\leq \frac{N}{2}+n.
$$
Let us consider the set $G_{n}\setminus W_{i_0}$. It is clear that all the paths between $W_{i_0}$ and $G_{n}\setminus W_{i_0}$ have to intersect the path $\gamma_1\cup \gamma_{i_0}$. The traffic between $W_{i_0}$ and $G_{n}\setminus W_{i_0}$ is equal to $|G_{n}\setminus W_{i_0}||W_{i_0}|/2$. Since 
$$
\frac{N}{2}-n\leq |G_{n}\setminus W_{i_{0}}|\leq \frac{N}{2}
$$
which satisfies
\begin{equation}
\frac{N}{4}\Bigg(\frac{N}{2}-n\Bigg) \leq \frac{|G_{n}\setminus W_{i_0}|\cdot|W_{i_0}|}{2}\leq \frac{N}{4}\Bigg(\frac{N}{2}+n\Bigg).
\end{equation}
Since this traffic has to pass through $\gamma_1\cup \gamma_{i_0}$ at some point then there exists at least one node $v$ in $\gamma_1\cup \gamma_{i_0}$ with load at least 
$$
T_{n}(v)\geq \frac{N}{4}\Big(\frac{N}{2}-n\Big)\frac{1}{2n}=\frac{N^2}{16n}-\frac{N}{8}
$$
(see Figure~1, right) proving our claim.
\end{proof}
Our first claim follows a corollary of the above theorem.
\begin{cor}
Let $G$ be an infinite planar graph with exponential growth. Then for $n$ sufficiently large there exists a node $v\in G_{n}$ such that
$$
T_{n}(v)\geq C\frac{N^2}{\log(N)}
$$
where $C$ is a constant independent on $n$.
\end{cor}

\section{Load on a General Small World Graph}\label{sec3}

It turns out that the planarity property is essential for the existence of highly congested nodes
proven above. 
We now show that in contrast, when $G$ is not planar, congestion can be made
to approach $O(N)$.  More explicitly, given $\epsilon>0$ 
there exists infinite graphs with exponential growth with uniformly bounded degree
such that for every node $v\in G_{n}$ and $n$ sufficiently large, $T_{n}(v)\leq N^{1+\epsilon}$.

\par Before providing such a construction let us show the following Lemma.

\begin{lem}
Let $G$ be an infinite graph and let $G_{n}$ be the ball of radius $n$ centered at $x_{0}$ as before. Assume moreover that $\sup_{v\in G} \mathrm{deg}(v)\leq \Delta<\infty$. Then for every $v\in G_{n}$ the following holds
\begin{equation}
T_{n}(v)\leq \Delta^2(\Delta-1)^{D-2}\cdot D^2
\end{equation}
where $D=\mathrm{diam}(G_{n})$.
\end{lem}

\begin{proof}
Let $v\in G_n$ and define $S_{k}:=\{x\in G_n: d(v,x)=k\}$. Then it is clear that $G_{n}=\{v\}\cup\bigcup_{p=1}^{D}S_{p}$ and moreover $|S_{k}|\leq \Delta(\Delta-1)^{k-1}$. Therefore,
$$
T_{n}(v)\leq \sum_{k+l\leq D}{|S_{k}|\cdot|S_{l}|}
$$
where the inequality is coming from the fact that if $k+l>D$ then the geodesic path between a node in $S_{k}$ and a node in $S_{l}$ does not pass through $v$. Hence,
$$
T_{n}(v)\leq \sum_{k+l\leq D}{\Delta^2(\Delta-1)^{k+l-2}}\leq \Delta^2(\Delta-1)^{D-2}D^2.
$$
\end{proof}
\par We state the following result due to Bollobas for completeness.

\begin{theorem}[\cite{BVega}]\label{Boll}
Given $r\geq 3$ for $N$ sufficiently large a random $r$--regular graph $X$ with $N$ nodes has diameter at most  
$$
\mathrm{diam}(X)\leq \log_{r-1}(N) + \log_{r-1}(\log_{r-1}(N))+C
$$
where $C$ is a fixed constant depending on $r$ and independent on $N$.
\end{theorem}

Now we are ready to show the construction of a small world graph with small congestion. Let $\epsilon>0$ and consider $G$ a infinite $k$-regular tree where the value of $k$ will be chosen later. Denote by $x_0$ the root of $G$ and $G_{n}=B(x_0,n)$ as before. Let $\widehat{G}_{n}$ be the graph constructed by connecting all the nodes in the spheres $S_{p}$
$$
S_{p}:=\{x\in G\,\,:\,\,d(x_0,x)=p\}
$$ 
by a $k$-regular random graph for every $1\leq p\leq n$. Then $|S_{p}|=k(k-1)^{p-1}$ and it is clear that using Theorem \ref{Boll} we have that
$$
\mathrm{diam}(\widehat{G}|S_{p})\leq p + \log_{k-1}(p) + C_{k}.
$$
Note that $|\widehat{G}_n|=|G_{n}|$ since we are not adding new nodes. It is not difficult to see that 
\begin{equation}\label{diam}
\mathrm{diam}(\widehat{G}_n)\leq \max_{t\leq s\leq n} \Big\{\mathrm{diam}(\widehat{G}|S_{t})+s-t\Big\} \leq n+\log_{k-1}(n)+C_{k}.
\end{equation}
Therefore, using the previous Lemma we see that for every node $v\in \widehat{G}_n$  
\begin{equation}
T_{n}(v)\leq (n+\log_{k-1}(n)+C_{k})^2\cdot (2k-1)^{n+\log_{k-1}(n)+C_{k}}.
\end{equation}
Therefore,
\begin{equation}
\frac{\log T_{n}(v)}{\log |G_{n}|}\leq \frac{(n+\log_{k-1}(n)+C_{k})\log(2k-1) + 2\log(n+\log_{k-1}(n)+C_{k})}{n\log(k-1)} 
\end{equation}
and hence 
\begin{equation}
\lim_{n\to\infty} \frac{\log T_{n}(v)}{\log |G_{n}|}=\frac{\log(2k-1)}{\log(k-1)}.
\end{equation}
By taking $k$ sufficiently large depending on $\epsilon$ we see that $T_{n}(v)\leq N^{1+\epsilon}$ for $n$ sufficiently large.

\section{The Existence of a Core, Hyperbolicity and Remetrization}

In \cite{NS2, NS1, Tucci1, Tucci2, H0, H3, H4, jonck}, it has been shown that $\delta$-hyperbolicity implies the existence of a core, that is, a non-empty set of nodes whose load scale as $O(N^2)$ under geodesic routing. In Section \ref{sec2}, we proved that planar graphs with exponential growth cannot avoid congestion of order $O(N^2/\log(N))$ no matter how the routing is performed. On the other hand, we observed in Section \ref{sec3} that exponential growth alone is not sufficient to guarantee the existence of such highly congested nodes. Thus, unlike the small world property, 
$\delta$-hyperbolicity is a sufficient
condition for a network to have highly congested nodes.  The reverse need not be true,
however. It is not difficult to construct non-hyperbolic graphs in which load scales as $O(N^2)$. 
For instance,
two square grids in two vertical planes separated by a single horizontal link joining their origins.
It is even possible to construct small world graphs with $O(N^2)$ load which are not hyperbolic.
Let $T_3$ be the 3-regular infinite tree and let $G=T_{3}\times \Z$. The graph $G$ is not Gromov
hyperbolic since it has $\Z^{2}$ as a sub-graph.  Yet, $(x_{0},0)$ has traffic of order $O(N^2)$ 
where $x_{0}$ is the root of $T_{3}$. 
It is interesting to note that even tough the graph $G$ is not hyperbolic, it has $T_{3}$ as a 
sub-graph. Examples of small world graphs with load of order $O(N^2)$ appear to include 
hyperbolic sub-graphs. We do not know if this is always true but it seems likely since exponential
growth implies existence of an exponentially growing tree sub-graph (e.g., its spanning tree).

\par We next explore what happens to hyperbolicity when we apply remetrization. More specifically, assume a metric 
graph $G$ where each edge $e$ has an associated non-negative distance $d_e$ that satisfies the triangle inequality. 
We modify each edge distance by a factor $0<w_e<\infty$ so that
the new length of the edge $e$ is $w_ed_e$. We also require that the coefficients $\{w_e\}_{e\in E}$ 
are chosen in such a way that the new edge distances continue to satisfy the triangle inequality and 
thus constitute a metric. 

\par To determine if $O(N^2)$ scaling of congestion persists after remetrization, let us start with a $\delta$-hyperbolic
graph and modify the edge metric according to the above scheme.  Does remetrization 
ensure another $\delta'$-hyperbolic graph?  We show below that this is not the case and thus
remetrization can significantly affect the congestion scaling in the graph, unless the
weights $w_e$ are bounded away from zero and infinity.

\begin{theorem}
For any $\delta$-hyperbolic graph, $\delta>0$, remetrization can change its scaling of congestion unless the
(remetrization) multipliers are uniformly bounded away from zero and infinity.
\end{theorem}

\begin{proof}
\par We shall prove the result for regular hyperbolic grids embedded in $\mathbb{H}^n$ and then appeal to 
the quasi-isometry of all $\delta$-hyperbolic graphs with these reference graphs (see \cite{bonk}) to 
complete the proof.  To simplify exposition, we focus
on dimension 2 only, since the argument carries through similarly for higher dimensions.
Let $G = H_{p,q}$ be a regular tessellation of the Poinc\'are disk with $1/p+1/q<4$, $p,q<\infty$. $H_{p,q}$ may be viewed as a (hyperbolic) grid where each node has (the same) degree $q$ and each face has (the same) $p$ sides, Figure~2 
depicts $H_{3,7}$. Note that in the case that $p=\infty$ the graph $G$ is a $q$-regular tree and is thus $0$-hyperbolic regardless of any metrization. Let $x_0$ be the node at the center of the disk and let $S_{k}$ be the set of nodes in $G$ at distance $k$ from $x_0$. Note that the sub-graph induced by the set $S_{k}$ is a cycle $C_{N(k)}$ with $N(k)$ nodes. Let us denote this graph also by $S_{k}$. It is not difficult to see that there exists a sequence $\{w_{k}\}_{k\geq 1}$ such that $w_{k}\to 0$ exponentially fast so that if we remetrize every edge in $S_{k}$ by the constant $w_{k}$ then the induced graph is not hyperbolic since it will be quasi-isometric to the Euclidean grid $\Z^2$. It was observed in \cite{NS2} and then proved in \cite{Tucci1}, that the nodes in $\Z^{q}$ have congestion of the order $O(N^{1+1/q})$. Therefore, the nodes in the new graph have a congestion of the order $O(N^{3/2})$. 

We observe that the above construction used arbitrarily small weights. More precisely, given $\epsilon>0$ there are infinitely many weights in this construction such that $w_{e}<\epsilon$. It is not hard to see that the same construction is possible with arbitrarily large 
weights instead of small weights. However, if we restrict these weights so that there exist positive constants $k$ and $K$ such that 
\begin{equation}\label{eeqq}
k\leq w_{e}\leq K \quad\quad\text{for all $e\in E$},
\end{equation}
then the original and the remetrized graphs are indeed quasi-isometric. Therefore, by a result of Gromov,
see \cite{Gromov, bridson}, if one graph is hyperbolic so is the other and thus $O(N^2)$ is
unaffected by the said change of metric.
\end{proof}

\begin{figure}[!Ht]\label{Fig2}
  \begin{center}
    \includegraphics[width=2.0in]{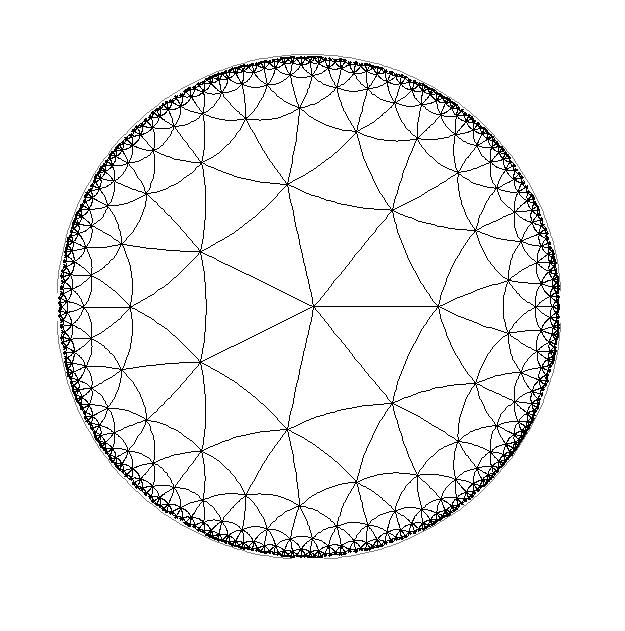}
    \caption{$H_{3,7}$ hyperbolic grid.}
  \end{center}
\end{figure}

\noindent {\it Acknowledgement}. This work was funded by NIST Grant No. 60NANB10D128.

\end{document}